\documentclass[12pt]{article}      
\usepackage{graphicx}

\usepackage{graphicx}
\usepackage{url}
\usepackage{color}
\definecolor{maroon}{rgb}{.69,.188,.376}
\definecolor{darkgreen}{rgb}{0,.5,0}
\definecolor{darkblue}{rgb}{0,0,.5}
\definecolor{magenta}{rgb}{1,0,1}

\usepackage[mathscr]{euscript}		

\usepackage{dsfont}

\usepackage{psfrag}			
\usepackage[colorlinks=true]{hyperref}
\hypersetup{pdftex, colorlinks=true, linkcolor=maroon, citecolor=maroon,
  filecolor=blue,urlcolor=blue}
\usepackage{amsmath}
\usepackage{amsthm}

\usepackage{amssymb}

\usepackage{caption} 

\usepackage{anysize}

\usepackage{enumerate} 
\usepackage{enumitem}

\usepackage[top=.95in, bottom = 1 in, left=1in, right = .6in]{geometry}

\newtheorem{theorem}{Theorem}
\newtheorem{lemma}{Lemma}[section]

\newtheorem{proposition}{Proposition}[section]
\newtheorem{remark}{Remark}[section]

\newtheorem{example}{Example}[section]

\numberwithin{equation}{section}

\definecolor{Red}{rgb}{1,0,0}
\definecolor{Blue}{rgb}{0,0,1}
\definecolor{Olive}{rgb}{0.41,0.55,0.13}
\definecolor{Yarok}{rgb}{0,0.5,0}
\definecolor{Green}{rgb}{0,1,0}
\definecolor{MGreen}{rgb}{0,0.8,0}
\definecolor{DGreen}{rgb}{0,0.55,0}
\definecolor{Yellow}{rgb}{1,1,0}
\definecolor{Cyan}{rgb}{0,1,1}
\definecolor{Magenta}{rgb}{1,0,1}
\definecolor{Orange}{rgb}{1,.5,0}
\definecolor{Violet}{rgb}{.5,0,.5}
\definecolor{Purple}{rgb}{.75,0,.25}
\definecolor{Brown}{rgb}{.75,.5,.25}
\definecolor{Grey}{rgb}{.7,.7,.7}
\definecolor{Black}{rgb}{0,0,0}

\newcommand{\ignore}[1]{{}}

\def\st{, \,}

\def\g{\gamma}

\def\hcap{{\rm hcap}}

\date{\today}

\begin{document}

\title{Smoothing of Boundary Behaviour in Stochastic Planar Evolutions}

\author{Atul Shekhar%
  \thanks{Lindstedtsv\"agen 25, Royal Institute of Technology,
         Stockholm, Sweden.
    Email: \url{atuls@kth.se}}
}

\maketitle
\begin{abstract}

Motivated by the study of trace for Schramm-Loewner evolutions, we consider evolutions of planar domains governed by ordinary differential equations with holomorphic vector fields $F$ defined on the upper half plane $\mathbb{H}$. We show a smoothing effect of the presence of noise on the boundary behaviour of associated conformal maps. More precisely, if $F$ is H\"older, we show that evolving domains vary continuously in uniform topology and their boundaries are continuously differentiable Jordan arcs. This is in contrast with examples from deterministic setting where the corner points on the boundary of domain $F(\mathbb{H})$ may give rise to corner points on the boundaries of corresponding evolving domains.

\end{abstract}



\section{Introduction and Results.}

In this article we consider evolutions of planar domains governed by ordinary differential equations (ODEs) with holomorphic vector fields. Let $\mathbb{H}:= \{z \in \mathbb{C}| Im(z) > 0\}$ be the upper half plane and $F: \mathbb{H} \to \mathbb{H}$ be a holomorphic function. Let $U:[0,\infty)\to \mathbb{R}$ be a continuous real valued function with $U_0=0$. One can construct a family of simply connected planar domains by letting $\mathbb{H}$ evolve under the flow of the equation  
\begin{equation}\label{model-eqn}
dZ_t = F(Z_t)dt + dU_t, \hspace{2mm} Z_s = z \in \mathbb{H}.
\end{equation}

For $(s,t)\in \Delta := \{(s,t)| 0\leq s \leq t < \infty\}$, let $\varphi(s,t,z)$ denote the value of solution $Z$ to equation \eqref{model-eqn} at time $t$ solved with the initial condition $Z_s = z\in \mathbb{H}$. Note that since $F$ is holomorphic on $\mathbb{H}$, viewing \eqref{model-eqn} as a time dependent ODE, solution $Z$ is uniquely well defined till the time the solution escapes $\mathbb{H}$. Also, since $Im(F(z)) >0$ for $z \in \mathbb{H}$ and $U$ is real valued, $Im(Z_t)$ is increasing in time $t$ and thus $Z_t$ stays in $\mathbb{H}$. Assuming that the solution $Z$ doesn't blow to infinity in finite time, $\varphi(s,t,z)$ is well defined for all $(s,t)\in \Delta$ and $z\in \mathbb{H}$. We will call $\varphi: \Delta \times \mathbb{H} \to \mathbb{H}$ as the flow associated to equation \eqref{model-eqn}. Further note that from the uniqueness of solution to ODEs, it follows easily that for each fixed $(s,t)\in \Delta$, $\varphi(s,t,.)$ is an injective holomorphic map defined on $\mathbb{H}$. The planar evolution associated to $(F,U)$ is then defined to be the family $\mathcal{H}= \{H_{s,t}\}_{(s,t)\in \Delta}$ of simply connected domains given by $H_{s,t}:= \varphi(s,t, \mathbb{H})$. Note that $H_{s,t}\subset H_{u,t}\subset \mathbb{H}$ for all $s\leq u\leq t$. It however doesn't satisfy any inclusion property in the $t$ variable. We are interested in understanding qualitative properties of domains in $\mathcal{H}$. Following questions arising naturally in this context:

\begin{enumerate}
\item Are the boundaries of domains in family $\mathcal{H}$ locally connected, i.e. can it be parameterized by a continuous function? If yes, are these boundaries also Jordan arcs?

\item How regular are the boundaries of domains in $\mathcal{H}$? 

\item Do boundaries of domains in $\mathcal{H}$ vary continuously in uniform topology as $s,t$ vary? 

\end{enumerate}

Motivation for considering above questions come from the study of trace in the theory of Loewner chains. When $F$ is given by the inversion of the upper half plane, i.e. $F(z) = \frac{-2}{z}$, equation \eqref{model-eqn} is known as reverse time Loewner differential equation driven by driving function $U$, also see interesting related works \cite{general-slit-chains},\cite{nam-gyu-kang} which considers more general vector fields $F$. Since our main result doesn't apply to Loewner chains and is related to it only on a heuristic level, we postpone an introduction to the theory of Loewner chains and its relation to \eqref{model-eqn} till section \ref{conclusion}. The main goal of this paper is to exhibit a clear distinction in the properties of planar evolution $\mathcal{H}$ when $U$ is a deterministic and regular driver versus when $U$ is a random and irregular driver. We will show that random and irregular drivers have some inbuilt smoothing effects which can not be expected when $U$ is a generic deterministic regular driver. We thus argue that the fundamental reason behind the existence of trace for deterministic and random Loewner chains are in fact different, see section \ref{trace-problem}.\\

We will restrict ourselves to H\"older continuous functions $F$. Recall that for $\alpha \in (0,1]$, $F$ is locally $\alpha$-H\"older if for all bounded subsets $A\subset \mathbb{H}$, 
\[ ||F||_{\alpha, A}:=\sup_{z,w \in A} \frac{|F(z)- F(w)|}{|z-w|^{\alpha}} < \infty.\] 
 
$F$ is called globally $\alpha$-H\"older if $||F||_{\alpha}:= \sup_{A\subset \mathbb{H}}||F||_{\alpha, A} < \infty$. H\"older holomorphic functions typically arise as Riemann uniformizing maps of simply connected domains whose boundary may have corners but has no cusps. Such domains are called H\"older domains, examples include domains bounded by various fractal curves such as von Koch snowflake. Riemann uniformizing maps of domains $H_{s,t}$ is by definition given by conformal maps $\varphi(s,t,.)$. Boundaries of domains $H_{s,t}$ can be thus understood by studying the boundary behaviour of maps  $\varphi(s,t,.)$ which in turn is intricately connected to the boundary behaviour of $F$. Our main result is the following theorem answering questions $(a)$-$(c)$ when $F$ is H\"older and $U$ is a sample path of standard Brownian motion $B$ defined on a filtered probability space $(\Omega,\{\mathcal{F}_t\}_{t\geq 0}, \mathcal{F}, \mathbb{P})$. Domains in the corresponding stochastic evolution $\mathcal{H}$ will in fact have continuously differentiable boundaries even if boundary of $F(\mathbb{H})$ may have corners (we however do not require $F$ to be injective). We denote the complex derivative of a holomorphic function $G(z)$ by $G'(z)$. $\overline{\mathbb{H}}= \mathbb{H}\cup \mathbb{R}$ is the closed upper half plane.

\begin{theorem}\label{main-thm}
Let $F$ be an uniformly bounded and globally $\alpha$-H\"older holomorphic function. Then almost surely for all $\omega \in \Omega$ and $U=B(\omega)$,

\begin{enumerate}[label=(\roman*)]
\item $\varphi$ extends continuously to $\Delta \times \overline{\mathbb{H}}$. In fact, $\varphi$ is jointly $\eta$-H\"older on compact subsets of $\Delta \times \overline{\mathbb{H}}$ for all $\eta < \frac{1}{2}$.  Further, for all $s\leq t$, $\varphi(s,t,.)$ is injective on $\overline{\mathbb{H}}$. In particular, boundaries of domains in $\mathcal{H}$ are continuously varying Jordan arcs.  

\item $\varphi'$ extends continuously to $\Delta \times \overline{\mathbb{H}}$ and for each $s\leq t$, $\varphi'(s,t,.)$ is $\eta$-H\"older on compact subsets of $\overline{\mathbb{H}}$ for all $\eta < \alpha$.

\end{enumerate}

\end{theorem}

\begin{remark}
We have assumed $F$ to be uniformly bounded only to guarantee that the solution to \eqref{model-eqn} doesn't explode in finite time and the flow $\varphi$ is globally well defined. This assumption can be replaced by Osgood-type conditions on $F$, e.g. $F(z)= O(|z|)$ as $z\to \infty$ which implies non-explosion of solution in finite time. The global H\"older condition on $F$ can also be replaced by local H\"older condition by using a localisation argument. For brevity, we will work with the global boundedness and global H\"older assumptions on $F$ and leave the details of the general case to interested readers.

\end{remark}

\begin{remark}
A classical result known as Kellogg-Warschawski Theorem states that Riemann uniformizing map of a Jordan domain has derivative which is $\eta$-H\"older on $\overline{\mathbb{H}}$ for some $\eta \in (0,1)$ if and only if the boundary of such a domain can be parametrized by a $C^{1+\eta}$ (i.e. differentiable with $\eta$-H\"older derivative) function defined on $[0,1]$. Thus, Theorem \ref{main-thm}-(ii) can be equivalently stated by saying that boundary of $H_{s,t}$ is of class $C^{1+\eta}$. Such an equivalence is no longer true for H\"older domains, i.e. for class $C^{\alpha}$, $\alpha \in (0,1)$. It is not enough for a domain with the boundary which can be parametrized by a H\"older function to be a H\"older domain and one requires the boundary to be H\"older in its conformal parametrization, see Privalov-Hardy-Littlewood Theorem in \cite[Chapter $2$]{garnett-marshall} and \cite[Chapter $3$]{pommerenke}.
\end{remark}

\begin{example}
To emphasize on the smoothing effect of noise term $B$ in Theorem \ref{main-thm}, consider the case of $F(z)= z^{\alpha}$ for $\alpha \in (0,1)$. It is easy to check that $F$ is $\alpha$-H\"older. The function $F$ maps $\mathbb{H}$ to the infinite sector formed by rays $\{y=\tan(\pi \alpha)x, x \geq 0\}$ and $\{x\geq 0\}$ meeting at origin and $F(\mathbb{H})$ has a corner at origin. If we consider the equation $dZ_t = F(Z_t)dt$ without the noise term $B$, the flow $\varphi$ can be explicitly solved and is given by 
\[\varphi(s,t,z) = \{(1-\alpha)(t-s) + z^{1-\alpha}\}^{\frac{1}{1- \alpha}}.\]  

Clearly, $\varphi'$ doesn't extend continuously to $0$ for all $s<t$ and domains $\varphi(s,t, \mathbb{H})$ will in fact have a corner at the point $\varphi(s,t,0)$. Theorem \ref{main-thm} shows that in the presence of noise, such corners get smoothed out.  

\end{example}

\begin{example}
Various examples of random simply connected domains can be constructed with different choices of holomorphic functions $F: \mathbb{H}\to \mathbb{H}$. Such functions are also called Nevanlinna-Herglotz functions and are completely characterized by the representation formula 
\[F(z) = C + Dz + \int_{\mathbb{R}} \biggl(\frac{1}{x-z} - \frac{x}{1+ x^2}\biggr)d\mu(x),\]
where $C, D$ are reals with $D\geq 0$ and $\mu$ is a Borel measure on $\mathbb{R}$ satisfying $\int_{\mathbb{R}}\frac{d\mu(x)}{1+x^2} < \infty$. 

More examples can be produced by iterating the above scheme. Denote $\varphi= \varphi^F$ to show its dependence on $F$. Starting with any function $F$, one can construct a sequence of holomorphic vector fields $\{F_{n}\}_{n\geq 0}$ with $F_0 = F$ and $F_{n+1}(z)= \varphi^{F_n}(0,1,z)$. In the special case when $F(z)= \frac{-2/\kappa}{z}$, $\kappa >0$, $F_1(\mathbb{H})$  is the complement of Schramm-Loewner-Evolution SLE$_{\kappa}$ curve in $\mathbb{H}$ which clearly has a fractal boundary. However, when $\kappa \neq 4$, it was proven in \cite{BasicSLE} that $F_1(\mathbb{H})$ is a H\"older domain and Theorem \ref{main-thm} implies that the second iteration $F_2(\mathbb{H})$ will in fact have continuously differentiable boundary. 
 
\end{example}

%

Theorem \ref{main-thm} has its root in earlier works on the phenomena of regularization by noise, see \cite{davie},\cite{fgp},\cite{beck-et-el},\cite{Mario-thesis} and references there in. A similar result as Theorem \ref{main-thm} was proven in \cite[Theorem $5$]{fgp} using partial differential equation (PDE) techniques. A key difference between \cite[Theorem $5$]{fgp} and Theorem \ref{main-thm} is that though the equation \eqref{model-eqn} is a two dimensional system of real valued equations, a one dimensional noise term $B$ is enough to produce smoothing effects. The fact that $F$ is complex differentiable on $\mathbb{H}$ is very crucial in our case so that the set where $F$ is irregular is indeed one dimensional. We however believe that PDE techniques from \cite{fgp} can be used to give a different proof of Theorem \ref{main-thm}, see section \ref{different-proof} for a brief sketch of their approach and how to adapt it to our case. Our proof of Theorem \ref{main-thm} is purely based on stochastic analysis and it will even allow one to consider non-H\"older functions $F$ satisfying certain conditions. For such conditions on $F$, we refer the interested reader to \cite{abs} where we have used the same technique as in this paper but in a slightly different context. \\

To prove Theorem \ref{main-thm}, we will consider \eqref{model-eqn} starting from a point on $\mathbb{R}= \partial\mathbb{H}$ interpreted as a stochastic differential equation (SDE). We will establish the existence and uniqueness of strong solution to such SDEs. But since unique strong solutions are only well defined upto a set of $\mathbb{P}$ measure $0$, we require some additional arguments to prove the almost sure joint continuity of $\varphi$. Continuity of $\varphi'$ will be established by obtaining an exponential identity for difference of solutions to \eqref{model-eqn} starting from two different points, see Proposition \ref{exp-identity} for details. \\

The organization of this article is as follows. In section \ref{prem-lemmas}, we establish some preparatory lemmas used in the proof of Theorem \ref{main-thm}. In section \ref{strong-solution}, we prove the strong uniqueness of solution to \eqref{model-eqn} starting on $\mathbb{R}$. Section \ref{proof-of-main-thm} contains the proof of Theorem \ref{main-thm}. We close the article with some concluding remarks in section \ref{conclusion}. \\

{\bf Acknowledgements:} I would like to thank Siva Athreya, Suprio Bhar, Peter Friz, Fredrik Viklund and Yilin Wang for various discussions and their fruitful comments. I acknowledge the financial support in the form of KTH visiting researcher scholarship through research grant of Fredrik Viklund from G\"oran Gustafsson foundation.

\section{Some Preliminary Lemmas.} \label{prem-lemmas}

\subsection{It\^o formula in complex $1$-dimension.} \label{ito-formula}

We will use It\^o formula as a fundamental tool in our proof. We recall it here in a non-standard form. We will be interested in $\overline{\mathbb{H}}$ valued semimartingales of form $W_t = L_t + B_t $ where $L$ is a $\overline{\mathbb{H}}$ valued almost surely continuous and bounded variation process and $B$ is real valued standard Brownian motion both adapted to the filtration $\{\mathcal{F}_t\}_{t\geq 0}$. $W$ is actually a two dimensional semimartingale with the real valued martingale component, but we will view $W$ as a complex $1$-dimensional process and use the complex field multiplication to define quadratic variation process $[W]$ of $W$ in the usual sense, i.e. 
\[[W]_t := \lim\limits_{|\mathcal{P}| \to 0} \sum_{[u,v] \in \mathcal{P}} (W_v - W_u)^2,\]
where $\mathcal{P}$ denote partitions of $[0,t]$ and the limit is taken in probability. Clearly, since $L$ is of bounded variation and $B$ is real valued, $[W]_t = t$. It\^o formula gives the semimartingale decomposition for processes of form $G(B_t)$ where $G$ is a continuously twice real differentiable function. Since we are working with complex field operations, we naturally consider complex differentiable or holomorphic functions $G$ and compute the semimartingale decomposition of process $G(W_t)$ as sum of a $\mathbb{C}$ valued martingale and a bounded variation process.

\begin{lemma}\label{complex-ito-formula}
Let $G: \overline{\mathbb{H}} \to \mathbb{C}$ be a holomorphic function, i.e. $G$ is holomorphic in an open set containing $\overline{\mathbb{H}}$. Then almost surely,
\[G(W_t) = G(W_0) + \int_0^t G'(W_r)dW_r + \frac{1}{2}\int_0^t G''(W_r)d[W]_r.\]
\end{lemma}
 
\begin{proof}
The proof follows easily using the power series expansion of holomorphic functions locally around each point in $\overline{\mathbb{H}}$ and details are left to the reader. 
\end{proof}

\subsection{Linear differential equations in the complex plane.}\label{RS-integral}

We will use linear differential equations of form $dS_t = S_tdV_t$ as another basic tool in our proof. We first recall some standard facts about Riemann-Stieltjes integrals. Given continuous functions $M,N:[0,T] \to \mathbb{C}$ such that $M$ is of bounded variation, the Riemann-Stieltjes integral
 \[ \int_0^t M_rdN_r := \lim\limits_{|\mathcal{P}|\to 0}\sum_{\xi \in [u,v]\in \mathcal{P}}M_{\xi}(N_v - N_u)\]
is well defined. Further, the integration by parts formula holds: 
\begin{equation}
\int_0^t M_rdN_r + \int_0^t N_rdM_r = M_tN_t - M_0N_0.
\end{equation}
In particular, if $\{M^k\},\{N^k\}$ are sequences of functions such that $M^k \to M$ in bounded variation norm and $N^k \to N$ in supremum norm as $k \to \infty$, then \[ \int_0^.M_r^kdN_r^k \to \int_0^. M_rdN_r \hspace{2mm} \mbox{in supremum norm as $k \to \infty$.} \]\\
Next, given a continuous function $V:[0,T] \to \mathbb{C}$, consider equations of form 
\begin{equation}\label{linear-ODE}
S_t = S_0 + \int_0^t S_rdV_r,
\end{equation}
where we will a priori assume that $S:[0,T] \to \mathbb{C}$ is a continuous bounded variation function and the right hand side of equation \eqref{linear-ODE} is well defined as a Riemann-Stieltjes integral. We will crucially use the following lemma. 

\begin{lemma}\label{S=0}
Let S be a continuous bounded variation function and $V$ be a continuous function satisfying \eqref{linear-ODE}. Then, 
\begin{enumerate}[label=(\roman*)]
\item if $S_0 = 0$, then $S \equiv 0$.
\item if $S_0 \neq 0$, then $V$ is of bounded variation and $S_t = S_0 \exp[V_t- V_0]$.
\end{enumerate}
\end{lemma}

\begin{proof}
For the part $(i)$, let $Z(S)= \{t| S_t = 0\}$ denote the zero set of $S$ which contains $t=0$ by assumption. If $S$ is not identically zero, then $[0,T]\setminus Z(S)$ is a non empty open set and it can be written as countable union of disjoint open intervals. Thus, there exists an interval $(a,b)$ such that $S_a =0$ and $S_t \neq 0$ for all $t \in (a,b)$. Then it easily follows that for all $\epsilon >0$ small enough and $t\in (a+ \epsilon, b)$,
\[ \int_{a+ \epsilon}^t \frac{dS_r}{S_r} = V_t - V_{a+ \epsilon}.\]
Evaluating the real part both sides gives $\log|S_t| - \log|S_{a+ \epsilon}| = Re(V_t) - Re(V_{a+ \epsilon})$. Taking $\epsilon \to 0+$ gives a contradiction. Thus $S$ is identically zero. \\

For the part $(ii)$, if $S_0 \neq 0$, it follows from the proof of part $(i)$ that $S_t \neq 0$ for all $t$ and 
\[ V_t = V_0 + \int_0^t \frac{dS_r}{S_r}.\]
Thus $V$ is of bounded variation. Define $\tilde{S}_t = S_0 \exp[V_t- V_0]$. Then $\tilde{S}$ is of bounded variation and solves $d\tilde{S}_t = \tilde{S}_tdV_t$. So, $d\{S-\tilde{S}\}_t = \{S-\tilde{S}\}_t dV_t$ with $S_0-\tilde{S}_0 = 0$. It then follows from part $(i)$ that $S_t = \tilde{S}_t$ which finishes the proof. 
\end{proof}

\section{Uniqueness of strong solution to \eqref{model-eqn} starting on $\mathbb{R}$.} \label{strong-solution}

In this section we fix a $x \in \mathbb{R}$ ($x$ can also be a $\mathcal{F}_0$ measurable random variable) and consider the SDE 
\begin{equation}\label{x-SDE}
 X_t = x + \int_0^tF(X_r)dr + B_t,
\end{equation}
where $F: \mathbb{H} \to \mathbb{H}$ is bounded and $\alpha$-H\"older and solution $X$ is assumed to be $\overline{\mathbb{H}}$ valued. Since $F$ is H\"older, it extends continuously to $F: \overline{\mathbb{H}} \to \overline{\mathbb{H}}$ and the above equation is well defined. We next establish the existence and uniqueness of $\overline{\mathbb{H}}$ valued strong solution, also see \cite{KR} for a related work. 

\begin{proposition}\label{uniqueness}
There exists a unique (upto indistinguishability) continuous process $X$ adapted to the filtration $\{\mathcal{F}_t\}_{t\geq 0}$, taking values in $\overline{\mathbb{H}}$ and satisfying equation \eqref{x-SDE}.
\end{proposition} 

\begin{proof}
We first establish the uniqueness of solution $X$. Let $Y$ be another such solution. For $\theta \in [0,1]$, define $W^{\theta}_t = \theta X_t + (1-\theta)Y_t$. Note that $W^{\theta}$ is a $\overline{\mathbb{H}}$ valued semimartingale of the same form as considered in section \ref{ito-formula}. Denote $S_t = X_t - Y_t$. Then, 
\begin{equation}\label{S-identity}
 S_t = \int_0^t \{F(X_r) - F(Y_r)\}dr.
\end{equation}
Note that $S$ is of bounded variation. Since $\mathbb{H}$ is simply connected and $F$ is holomorphic on $\mathbb{H}$, there exists a holomorphic function $G$ (unique upto constant) defined on $\mathbb{H}$ such that $G'(z)= F(z)$. Since $F$ is continuous on $\overline{\mathbb{H}}$, $G$ also extends continuously to $\overline{\mathbb{H}}$. Let
\[V_t = 2\int_0^1\biggl\{G(W_t^{\theta}) - G(W_0^{\theta}) - \int_0^t F(W_r^{\theta})dW_r^{\theta}\biggr\}d\theta.\]
Note that $V$ is almost surely a continuous process. We claim that almost surely,
\begin{equation}\label{ito-tanaka}
\int_0^t \{F(X_r) - F(Y_r)\}dr = \int_0^t S_rdV_r,
\end{equation} 
where the right hand side is a Riemann-Stieltjes integral as discussed in section \ref{RS-integral}. Equation \eqref{ito-tanaka} is an instance of famously known It\^o-Tanaka trick. For proving it, introduce functions $F_y(z):= F(z + iy), G_y(z):= G(z+ iy)$ for $y>0$. Functions $F_y, G_y$ are defined and holomorphic in an open set (depending on $y$) containing $\overline{\mathbb{H}}$. Using Fubini Theorem and Lemma \ref{complex-ito-formula}, 
\begin{equation}\label{y-eqn}
\int_0^t \{F_y(X_r) - F_y(Y_r)\}dr = \int_0^t\int_0^1 F_y'(W_r^\theta)S_rd\theta dr = \int_0^t S_rdV_r^y,
\end{equation}
where 
\begin{align*}
V^y_t &= \int_0^t\int_0^1F_y'(W_r^\theta)d\theta dr \\
&= \int_0^1\int_0^tF_y'(W_r^\theta)drd\theta \\
&= 2\int_0^1\biggl\{G_y(W_t^\theta) - G_y(W_0^\theta) - \int_0^t F_y(W_r^\theta)dW_r^\theta\biggr\}d\theta.
\end{align*}
Note that $F_y, G_y$ converge uniformly on compacts subsets of $\overline{\mathbb{H}}$ to $F, G$ respectively as $y\to 0+$. Thus, from standard facts about It\^o stochastic integrals, almost surely $V^y$ converge uniformly to $V$ as $y\to 0+$.  Since $S$ is of bounded variation, using properties of Riemann-Stieltjes integral discussed in section \ref{RS-integral}, 
\[ \int_0^.S_rdV_r^y \to \int_0^. S_rdV_r \hspace{2mm} \mbox{uniformly as $y \to 0+$}.\]
Thus, taking $y\to 0+$ in equation \eqref{y-eqn} proves equation \eqref{ito-tanaka}. Finally, combining \eqref{S-identity},\eqref{ito-tanaka} and using Lemma \ref{S=0} implies $S\equiv0$ or $X\equiv Y$ which finishes the proof. \\
The existence of solution $X$ can be easily established using standard Picard iteration method and details are left to the reader.
\end{proof}

\section{Proof of Theorem \ref{main-thm}.} \label{proof-of-main-thm}

We complete the proof of Theorem \ref{main-thm} in this section by providing candidates for continuous extensions of $\varphi$ and $\varphi'$ to $\Delta\times \overline{\mathbb{H}}$. For $(s,t,z) \in \Delta\times \overline{\mathbb{H}}$, with a slight abuse of notation, we use the same symbol $\varphi(s,t,z)$ to denote the random variable defined by the value at time $t$ of the strong solution to \eqref{model-eqn} with the initial condition $Z_s= z$. Note that we have used Proposition \ref{uniqueness} to produce the unique strong solution and the joint law of random variables $\{\varphi(s,t,z)\}_{(s,t,z) \in \Delta\times \overline{\mathbb{H}}}$ is uniquely well defined. It also follows from Proposition \ref{uniqueness} that for all fixed $s\leq u \leq t$ and $z\in \overline{\mathbb{H}}$, almost surely 
\begin{equation}\label{flow-property}
\varphi(s,t,z) = \varphi(u,t, \varphi(s,u,z))\hspace{2mm}\mbox{and}\hspace{2mm}\varphi(s,s,z)=z.
\end{equation}
We call equation \eqref{flow-property} as the \textit{Flow property}. Define the following symbols similarly as in section \ref{strong-solution}. For $\theta\in[0,1]$, $(s,t)\in \Delta$ and $z,w\in \overline{\mathbb{H}}$,
\[W^{\theta}(s,t,z,w) := \theta\varphi(s,t,z)+ (1-\theta)\varphi(s,t,w), \]

\begin{align*}
I^\theta(s,t,z,w) := G(&W^{\theta}(s,t,z,w)) - G(W^{\theta}(s,s,z,w)) \\ 
&- \int_s^t F(W^{\theta}(s,r,z,w))(\theta F(\varphi(s,r,z)) + (1-\theta)F(\varphi(s,r,w)))dr,
\end{align*}
where $G'= F$. Further define, 
\[ J^\theta(s,t,z,w) := \int_s^tF(W^{\theta}(s,r,z,w))dB_r,\]
\begin{equation*}
I(s,t,z,w) := \int_0^1 I^\theta(s,t,z,w)d\theta, \hspace{2mm} J(s,t,z,w):= \int_0^1 J^\theta(s,t,z,w)d\theta
\end{equation*}
and \[V(s,t,z,w) := 2\{I(s,t,z,w) - J(s,t,z,w)\}.\]\\
The following Proposition is the key to Theorem \ref{main-thm}.

\begin{proposition}\label{exp-identity} Let $F$ be uniformly bounded globally $\alpha$-H\"older function. Then the following holds:
\begin{enumerate}[label=(\roman*)]
    \item For each fixed $s,t,z,w$, almost surely 
\begin{equation*}\label{derivative-identity}
\varphi(s,t,z) - \varphi(s,t,w) = (z-w) \exp[V(s,t,z,w)].
\end{equation*}
\item For $p \geq 2 $, there exists a constant $ C = C(p, F, T)$ depending only on $p,$ $F$ and $T$ such that for all $0\leq s \leq t \leq T$, $0\leq u\leq v\leq T$ and $z,w \in \overline{\mathbb{H}}$, 
\begin{equation*} \label{Lp-bound}
 \mathbb{E}[| \varphi(s,t,z) - \varphi(u, v, w) |^p] \leq  C (|s-u|^{\frac{p}{2}}+|t-v|^{\frac{p}{2}}+|z-w|^p).
\end{equation*}
\item For $p \geq \frac{2}{\alpha} $, there exists a constant $ C = C(p, F, T)$ depending only on $p$, $F$ and $T$ such that for all $0\leq s \leq t \leq T$, $0\leq u\leq v\leq T$ and $z, w,\tilde{z},\tilde{w} \in \overline{\mathbb{H}}$,
\begin{equation*} \label{stochastic-integral-continuous}\mathbb{E}[|J(s,t, z, w) - J(u, v, \tilde{z}, \tilde{w})|^p] \leq C( |s-u|^{\frac{p\alpha}{2}} + |t-v|^{\frac{p\alpha}{2}} + |z-\tilde{z}|^{p\alpha} + |w-\tilde{w}|^{p\alpha}).
\end{equation*}
  
  \end{enumerate}
\end{proposition}

\begin{proof}
\begin{enumerate}[label=(\roman*)]
\item[$(i)$] Let $S_t = \varphi(s,t,z) - \varphi(s,t,w)$. Then 
\[S_t = z-w + \int_s^t \{F(\varphi(s,r,z)) - F(\varphi(s,r,w))\}dr.\]
Similarly as in the proof of Proposition \ref{uniqueness}, it follows that almost surely,
\[\int_s^t \{F(\varphi(s,r,z)) - F(\varphi(s,r,w))\}dr = \int_s^tS_rdV(s,r,z,w)\]
Thus \[S_t = z-w + \int_s^tS_rdV(s,r,z,w)\]
and Lemma \ref{S=0} proves the claim.

\item[$(ii)$] We will use $\lesssim$ to denote inequalities upto multiplication by constants $C$ depending only on $p, F, T$ which may change from line to line. W.l.o.g. we can assume $s \leq u$. If $t \in [s, u]$, we use the bound   
  \begin{eqnarray*}
    \lefteqn{|\varphi(s,t,z) - \varphi(u,v,w)|}\\
    &=&\bigl| z - w + \int_s^t F(\varphi(s,r,z))dr - \int_{u}^{v}F(\varphi(u, r, w))dr + B_t - B_s - (B_{v} - B_{u}) \bigr| \\
&\lesssim& |z- w| + |t-s| + |v - u| + |B_t - B_{v}| + |B_{s} - B_{u}|.
\end{eqnarray*}
Since in this case $|t-s| \leq |s- u|$ and $|v - u| \leq |v - t|$, the claim follows using standard moment bounds on increments of Brownian motion.\\
When $t \geq u$, using the flow property, $\varphi(s,t,z) - \varphi(u, v, w) = \varphi(u, t, \varphi(s, u, z)) - \varphi(u, v, w)$. Since \[|\varphi(u, v, w) - \varphi(u, t, w)| \lesssim |t-v| + |B_{t} - B_{v}|,\] it is enough to get the desired moment estimate on $ |\varphi(u, t, \varphi(s, u, z)) - \varphi(u, t, w)|$. Using previous part $(i)$, 
\[ |\varphi(u, t, \varphi(s, u, z)) - \varphi(u, t, w)| = | \varphi(s, u, z)-w| \times |\exp[V(u,t,\varphi(s, u, z),w)]|.\]
Again, since $\varphi(s, u, z)$ is $\mathcal{F}_u$ measurable and
\[ | \varphi(s, u, z)-w| \lesssim |z-w| + |s-u| + |B_s-B_{u}|,\]
it suffices to show that 
\begin{equation}\label{exp-martingale-bound} 
\mathbb{E}\bigl\{|\exp[V(u,t,\varphi(s, u, z),w)]|^p\bigl|\mathcal{F}_u\bigr\} \lesssim 1.
\end{equation}
To this end, note that $V= 2(I-J)$ and 
\[ |I^\theta(u,t,\varphi(s, u, z),w)|   \lesssim |t-u| + |B_t -B_{u}|.\]

Fernique Theorem implies that if $q \lesssim 1$, then $\mathbb{E}\{\exp[q|B_t-B_u|]\} \lesssim 1$. Thus, using H\"older inequality, it suffices to prove that for $q\lesssim 1$, 
\[\mathbb{E}\bigl\{\exp[q|ReJ(u,t,\varphi(s, u, z),w)|]\bigl|\mathcal{F}_u\bigr\} \lesssim 1.\]

Note that $Re(J^\theta)$ is a martingale with $[Re(J^\theta)]_t \lesssim (t - u)$. By Dambis-Dubins-Schwarz martingale embedding theorem, $Re(J_t^\theta) = \tilde{B}_{[Re(J^\theta)]_t}$ for some another Brownian motion $\tilde{B}$. Using the Fernique Theorem again gives us the bound \eqref{exp-martingale-bound} which proves the claim.

\item[$(iii)$] 
W.l.o.g. we assume $s \leq u$. If $t \in [s,u]$, then using Fubini Theorem for stochastic integrals and Burkholder-Davis-Gundy inequality, 
\begin{align*} \mathbb{E}[|J(s,t, z, w)|^p] &= \mathbb{E}\biggl| \int_s^t\int_0^1 F(W^{\theta}(s,r,z,w))d\theta dB_r\biggr|^p \\ & \lesssim \mathbb{E}\biggl( \int_s^t\biggl\{\int_0^1 |F(W^{\theta}(s,r,z,w))|d\theta \biggr\}^2dr\biggr)^\frac{p}{2} \\
  & \lesssim |t-s|^{\frac{p}{2}}\\
& \lesssim |s-u|^{\frac{p}{2}}.
\end{align*}
Similarly, $\mathbb{E}[|J(u,v, \tilde{z}, \tilde{w})|^p]  \lesssim |t-v|^{\frac{p}{2}}$ which gives 
\[ \mathbb{E}[|J(s,t,z,w) - J(u,v, \tilde{z}, \tilde{w})|^p] \lesssim |s-u|^{\frac{p}{2}}+ |t-v|^{\frac{p}{2}}.\]
If $t \geq u$, then 
\[ J(s,t,z, w) = \int_s^{u} \int_0^1 F(W^{\theta}(s,r,z,w))d\theta dB_r + \int_{u}^t\int_0^1 F(W^{\theta}(s,r,z,w))d\theta dB_r,\]
and
\[ J(u,v,\tilde{z}, \tilde{w}) = \int_t^{v} \int_0^1 F(W^{\theta}(u,r,\tilde{z},\tilde{w}))d\theta dB_r + \int_{u}^{t} \int_0^1 F(W^{\theta}(u,r,\tilde{z},\tilde{w}))d\theta dB_r.\]
Thus using the previous step, 
\begin{align*}
\mathbb{E}[|J(s,t,z,w) - J(u,v, \tilde{z}, \tilde{w})|^p] & \lesssim |s-u|^{\frac{p}{2}}+ |t-v|^{\frac{p}{2}} + \Theta(s,t,z,w,u,v,\tilde{z},\tilde{w}),
\end{align*}
where $\Theta(s,t,z,w,u,v,\tilde{z},\tilde{w})$ is defined by
\[\Theta(s,t,z,w,u,v,\tilde{z},\tilde{w})= \mathbb{E}\biggl[ \biggl|\int_{u}^t\int_0^1 \{ F(W^{\theta}(s,r,z,w)) -  F(W^{\theta}(u,r,\tilde{z},\tilde{w}))\}d\theta dB_r\biggr|^p\biggr]. \]
Again using the Burkholder-Davis-Gundy inequality and previous part $(ii)$, 
\begin{align*}
\Theta(s,t,z,w,u,v,\tilde{z},\tilde{w}) &\lesssim \int_{u}^t\int_0^1 \mathbb{E}[|F(W^{\theta}(s,r,z,w)) -  F(W^{\theta}(u,r,\tilde{z},\tilde{w}))|^p]d\theta dr \\
& \lesssim \int_0^T\mathbb{E}[|\varphi(s,r,z) - \varphi(u,r,\tilde{z})|^{p\alpha }+ |\varphi(s,r,w)- \varphi(u,r, \tilde{w})|^{p\alpha}]dr \\
&  \lesssim | s- u|^{\frac{p\alpha}{2}} + |z- \tilde{z}|^{p\alpha} + |w-\tilde{w}|^{p\alpha},
\end{align*}
and putting together all the above different cases completes the proof.

\end{enumerate}
\end{proof}

\begin{proof}[Proof of Theorem \ref{main-thm}]
From the moment estimate obtained in Proposition \ref{exp-identity}-$(ii)$, it follows using Kolmogorov-Chentsov Theorem that the random field $\{\varphi(s,t,z)\}_{(s,t,z)\in \Delta\times \overline{\mathbb{H}}}$ has a continuous modification which is also almost surely $\eta$-H\"older on compact subsets of $\Delta\times \overline{\mathbb{H}}$ for all $\eta < \frac{1}{2}$. With an abuse of notation, we will use the same symbol $\varphi$ to denote its continuous version. Similarly, using Proposition \ref{exp-identity}-$(iii)$, the random field $\{J(s,t,z,w)\}_{(s,t,z,w)\in \Delta\times \overline{\mathbb{H}}\times \overline{\mathbb{H}}}$ also has a continuous modification which is almost surely $\eta$-H\"older on compact subsets of $\Delta\times \overline{\mathbb{H}}\times \overline{\mathbb{H}}$ for all $\eta < \frac{\alpha}{2}$. Again, we use the same letter $J$ to denote its continuous version. It follows that $V= 2(I-J)$ is almost surely continuous as well and Proposition \ref{exp-identity}-$(i)$ implies that almost surely for all $(s,t,z,w)\in \Delta\times \overline{\mathbb{H}}\times \overline{\mathbb{H}}$,
\[\varphi(s,t,z) - \varphi(s,t,w) = (z-w) \exp[V(s,t,z,w)].\]
This implies that $\varphi(s,t,.)$ is injective and continuously differentiable on $\overline{\mathbb{H}}$ with $\varphi'(s,t,z) = \exp[V(s,t,z,z)]$ completing the proof. 
\end{proof}

\section{An Informal Discussion.}\label{conclusion}
\subsection{Trace of Loewner chains.} \label{trace-problem} 

We first briefly recall some basics about chordal Loewner's theory in the upper half plane $\mathbb {H}$, see \cite{lawlerbook} for details.\\

Let $\gamma$ be a continuous injective curve 
from the compact time interval $[0,T]$ into $\mathbb{H}\cup \{0\}$ with $\g(0) =0$. Loewner's theory provides a way to encode the curve $\g$ via a real valued function $U$ which will be called the driving function or simply the driver of $\g$ which is defined as follows. Note that for each $t\geq 0$, $H_t := \mathbb{H}\setminus \g[0,t]$ is a simply connected domain and there exists a unique conformal map $g_t$ from the slit domain $H_t$ onto $\mathbb {H}$ satisfying the so called \textit{hydrodynamic normalization} given by $\lim_{z \to \infty} (g_t(z) - z  ) =  0$.  
The map $g_t$ is called the mapping-out function of the set $K_t := \gamma[0,t]$. Laurent series expansion of $g_t$ at infinity yields a non-negative constant $b_t$ depending on $K_t $ such that 
\[g_t(z) = z + \frac{b_t}{z} + O\biggl(\frac{1}{|z|^2}\biggr), \hspace{2mm} \mbox{as} \hspace{2mm} z \to \infty.\]

The constant $b_t$ is called the half-plane capacity of $K_t$ and denoted by $b_t = \hcap(K_t)$. It is easy to see that $t \mapsto \hcap (K_t)$ is continuously increasing. Thus it is possible to choose a parameterization of $\gamma$ so that $\hcap(K_t) = 2t$ for all $t \in [0,T]$. The mapping out function $g_t$ also admits a continuous extension to the boundary point $\gamma_t $ of the domain $H_t$. The driver $U$ is then defined by $U_t := g_t(\gamma_t)$ which can be easily shown to be a continuous real valued function.  The significance of the driver $U$ comes from the fact that it describes the evolution of the conformal maps $g_t(z)$ in variable $t$ via Loewner differential equation (LDE) given by  
\begin{equation}\label{LDE} 
\partial_t{g}_t(z) = \frac{2}{g_t(z)- U_t}, \hspace{2mm} g_0(z)= z. 
\end{equation} 

In fact one can also recover the curve $\gamma$ from $U$ as follows. For each $z\in \overline{\mathbb{H}}\setminus \{0\}$, let $[0,T_z)$ with $T_z \in (0, \infty]$ denote the maximal interval of existence of the unique solution to equation \eqref{LDE}. Also define $T_0 = 0$. Then 
\[ \gamma[0,t] = \{z \in \overline{\mathbb{H}} \st T(z) \leq t\}.\]

The procedure described above can also be naturally reversed. Given any continuous real valued function $U$ with $U_0=0$, define $g_t(z)$ for $z \in \overline{\mathbb{H}}\setminus \{0\}$ to be the solution of \eqref{LDE}. Let $T_z$ for $z\in \overline{\mathbb{H}}$ be similarly defined as above. Then 
\[ K_t := \{z \in \overline{\mathbb{H}} \st T(z) \leq t\}\]
defines an increasing family of compact sets in $\overline{\mathbb{H}}$. The family $K= \{K_t\}_{t \in [0,T]}$ is called the \textit{Loewner chain driven by} $U$. As in the previous case, $H_t :=\mathbb{H}\setminus K_t$ is simply connected and $g_t$ is the unique conformal map mapping $H_t$ to $\mathbb{H}$ satisfying hydrodynamic normalization. The Loewner chain $K$ also satisfies $\hcap(K_t)=2t$ and the so called \textit{conformal local growth property} meaning that the radius of $g_{t}(K_{t+s} \setminus K_t)$
 tends to $0$ as $s \to 0+$ uniformly with respect to $t$. However, in full generality, $K_t$ may not be locally connected and it cannot always be written as the image set $\gamma[0,t]$ for a continuous curve $\gamma$, e.g. logarithmic spirals, see \cite{LMR10}. Even if this is the case, the curve $\gamma$ may be non-simple and $K_t$ has to be described by filling the loops in the image $\gamma[0,t]$. We say that the Loewner chain $K$ driven by $U$ admits a trace or synonymously $U$ generates a trace if there exist a continuous curve $\g:[0,T]\to \overline{\mathbb{H}}$ such $\gamma_0=0$ and for all $t \in [0,T]$, $H_t$ is the unbounded component of $\mathbb{H}\setminus \gamma[0,t]$. We then call $\gamma$ is the trace of the Loewner chain $K$. \\

Denote $\hat{f}_t(z) = g_t^{-1}(z+ U_t)$. It can be easily shown that for each fixed $t\in [0,T]$, $\hat{f}_t(z)= h_t(z)$, where for $s\in [0,t]$, $h_s(z)$ is the solution to reverse time LDE given by 
\begin{equation}\label{reverse-LDE}
h_s(z) =z +  U_t - U_{t-s} + \int_0^s \frac{-2}{h_r(z)}dr.
\end{equation} 

Equation \eqref{reverse-LDE} is a special case of \eqref{model-eqn} driven by time reversal of $U$ with the vector field $F(z) = -2/z$ and above raised questions $(a)$-$(c)$ translates into the study of trace of Loewner chains. A natural question to ask is for which drivers $U$ does the associated Loewner chain $K$ admits a trace? It was proven in \cite{RM05},\cite{Lind} that if $U$ is $1/2$-H\"older with $||U||_{\frac{1}{2}}<4$, then $U$ generates a simple curve $\gamma$. In the random setting when $U_t = \sqrt{\kappa}B_t$, where $\kappa>0$ and $B$ is standard Brownian motion, the existence of trace $\gamma$ which are called Schramm-Loewner-Evolution (SLE$_{\kappa}$) was shown in \cite{BasicSLE} (for $\kappa\neq 8$) and \cite{SLE8} (for $\kappa =8$). Theorem \ref{main-thm} suggests that there is a clear distinction between existence of trace for $1/2$-H\"older drivers and Brownian drivers because Brownian drivers produce smoothing effects which is attributed to its high irregularity. Since H\"older condition on $U$ only measures its regularity by putting an upper bound on its modulus of continuity, $1/2$-H\"older drivers will typically not produce any smoothing effects. Heuristically speaking, the phenomena of existence of trace for Loewner chains is not just dependent on how regular the driver is but it is also dependent on its irregularity to some extent. In a recent work Catellier-Gubinelli \cite{CG} introduced the notion of $\rho$-irregularity of functions which captures such smoothing effects from a deterministic point of view. 
 
\subsection{A PDE approach to Theorem \ref{main-thm}.} \label{different-proof}

We briefly sketch a different approach based on the work of Flandoli-Gubinelli-Priola \cite{fgp} to prove Theorem \ref{main-thm}. The idea is to find an appropriate univalent map (i.e. injective and holomorphic) $\Phi: \mathbb{H} \to \mathbb{C}$ and consider the transformation $Y_t= \Phi(Z_t)$, where $Z$ solves \eqref{model-eqn} with $U=B$. By appropriately choosing $\Phi$, \eqref{model-eqn} can be transformed into a SDE for $Y$ whose coefficients have more regularity than that of $F$. The function $\Phi$ is chosen as the unique solution to the following second order ODE (which is actually a PDE in disguise because of the differentiation w.r.t. complex variable $z$)
\begin{equation}\label{PDE}
\frac{1}{2}\Phi''(z) + F(z) \Phi'(z) = \lambda (\Phi(z) -z), \hspace{2mm} \Phi(z) = z+ O(1) \hspace{2mm} \mbox{as} \hspace{2mm} z \to \infty, 
\end{equation} 
where constant $\lambda$ is appropriately chosen. In fact, $\Phi$ can be explicitly written as 
\begin{equation}
\Phi(z) = z+ \int_0^{\infty}\exp(-\lambda r)\mathbb{E}[F(\varphi(0,r,z))]dr.
\end{equation} 

An argument is required to verify that $\Phi$ is indeed an univalent function. Hadamard's global inverse function theorem was used in \cite{fgp} to verify the injectivity (for large enough $\lambda$) in their case. Testing univalency of holomorphic functions is more subtle and Hadamard's theorem will not apply in this case. However, there are various known criterion such as Becker's univalence criteria for such purposes, see \cite{garnett-marshall}.\\

One can easily check using It\^o formula that $Y_t = \Phi(Z_t)$ satisfies 
\begin{equation}\label{Y-eqn}
dY_t = \tilde{F}(Y_t)dt + \sigma(Y_t)dB_t , \hspace{2mm} Y_0= \Phi(z),
\end{equation} 
where $\tilde{F}(z) = \lambda(z - \Phi^{-1}(z))$ and $\sigma(z) = \Phi'(\Phi^{-1}(z))$. It is expected that function $\Phi$ has $C^{2+\alpha}$ regularity on $\mathbb{R}$, i.e. twice differentiable with $\alpha$-H\"older second derivative. Indeed, in \cite{fgp} where the authors deal with multidimensional Brownian motion, the second derivative term in \eqref{PDE} is replaced by the Laplacian and since $F$ is $\alpha$-H\"older, $2+ \alpha$ regularity of $\Phi$ is a consequence of Schauder's estimate from elliptic regularity theory. Schauder's estimate do not strictly apply in our case because of the $B$ is one dimensional and there is no Laplacian term. But, since $F$ is holomorphic on $\mathbb{H}$ and it is irregular only on $\mathbb{R}$, it is expected that $\Phi$ will indeed have $2+ \alpha$ regularity. Thus, functions $\tilde{F}, \sigma$ have $1+ \alpha$ regularity and classical results of Kunita \cite{kunita} on stochastic flows imply that equation \eqref{Y-eqn} has $C^{1+ \alpha}$ regular stochastic flow. Transforming \eqref{Y-eqn} back to \eqref{model-eqn} using $\Phi^{-1}$ gives a different proof of Theorem \ref{main-thm}.

\bibliographystyle{plain}

\end{document}